\documentclass[11pt]{amsart}
\usepackage{amssymb}
\usepackage{color}
\usepackage{epsfig}
\usepackage{url}
\usepackage{setspace}
\usepackage{pdflscape}
\usepackage{etoolbox}
\theoremstyle{plain}

\newtheorem{thm}{Theorem}[section]

\newtheorem{lem}[thm]{Lemma}

\newtheorem{rem}[thm]{Remark}
\newtheorem{ques}[thm]{Question}
\newtheorem{conj}[thm]{Conjecture}

\def\cal{\mathcal}
\def\bbb{\mathbb}
\def\op{\operatorname}
\renewcommand{\phi}{\varphi}

\newcommand{\N}{\bbb{N}}
\newcommand{\Z}{\bbb{Z}}
\newcommand{\Q}{\bbb{Q}}

\begin{document}
%
%
%
%
%
%
%
%
%

\title[Power values of sums of certain products of consecutive integers]{Power values of sums of certain products of consecutive integers and related results} \author{Szabolcs Tengely, Maciej Ulas}

\keywords{power values, product, consecutive integers, high degree Diophantine equations} \subjclass[2010]{11D41}

\begin{abstract} Let $n$ be a non-negative integer and put $p_{n}(x)=\prod_{i=0}^{n}(x+i)$. In the first part of the paper, for given $n$, we study the existence of integer solutions of the Diophantine equation
$$
y^m=p_{n}(x)+\sum_{i=1}^{k}p_{a_{i}}(x),
$$
where $m\in\N_{\geq 2}$ and $a_{1}<a_{2}<\ldots <a_{k}<n$. This equation can be considered as a generalization of the Erd\H{o}s-Selfridge Diophantine equation $y^m=p_{n}(x)$. We present some general finiteness results concerning the integer solutions of the above equation. In particular, if $n\geq 2$ with $a_{1}\geq 2$, then our equation has only finitely many solutions in integers. In the second part of the paper we study the equation
$$
y^m=\sum_{i=1}^{k}p_{a_{i}}(x_{i}),
$$
for $m=2, 3$, which can be seen as an additive version of the equation considered by Erd\H{o}s and Graham. In particular, we prove that if $m=2, a_{1}=1$ or $m=3, a_{2}=2$, then for each $k-1$ tuple of positive integers $(a_{2},\ldots, a_{k})$ there are infinitely many solutions in integers.

\end{abstract}

\maketitle

\section{Introduction}\label{sec1}

Let $\N$ denote the set of positive integers, $\N_{0}$ the set of non-negative integers and $\N_{\geq k}$ will denote the set of non-negative integers $\geq k$. For $n\in\N_{0}$ we write
$$
p_{a}(x)=\prod_{i=0}^{a}(x+j).
$$
Moreover, we define the set
$$
A_{n}=\{(a_{1},\ldots,a_{k})\in\N_0^{k}:\;a_{i}<a_{i+1}\;\mbox{for}\;i=1,2,\ldots k-1,\;a_{k}<n\;\mbox{and}\; k\in\{1,\ldots,n-1\}\}.
$$
For given $m\in\N_{\geq 2}$ and $T=(a_{1},\ldots,a_{k})\in A_{n}$ we consider the Diophantine equation
\begin{equation}\label{maineq}
y^m=g_{T}(x), \quad\mbox{where}\quad g_{T}(x):=p_{n}(x)+\sum_{i=1}^{k}p_{a_{i}}(x).
\end{equation}
The cardinality of $A_{n}$ is $2^{n}-1,$ hence for a given $m$ we deal with $2^{n}-1$ Diophantine equations.


Let us observe that equation (\ref{maineq}) can be seen as a generalization of the classical Diophantine equation
\begin{equation}\label{erse}
y^{m}=p_{n}(x).
\end{equation}
The literature of this type of Diophantine equations is very extensive. Erd\H{o}s \cite{er} and Rigge \cite{r} independently proved that a product of two or more consecutive integers is never a perfect square. Erd\H{o}s and Selfridge \cite{ErSel} proved that the above equation has no solutions in integers $(x,y,m,n)$ satisfying the conditions $n\geq 1, m\geq 2$ and $y\neq 0$.
A difficult conjecture states that even a product of consecutive terms of an arithmetic progression of length at least four and difference at least one is never a perfect power. Euler proved (see \cite{Dickson} pp. 440 and 635) that a product of four terms in arithmetic progression is never a square. Obl\'ath \cite{Oblath} obtained a similar statement in case of five terms. Bennett, Bruin, Gy\H{o}ry and Hajdu \cite{BBGyH} extended this result to the case of arithmetic progressions having at most 11 terms. Gy\H{o}ry, Hajdu and Pint\'{e}r \cite{GHP} extended these results for at most 34 terms
Hirata-Kohno, Laishram, Shorey and Tijdeman \cite{hklst} completely solved the Diophantine equations related to square products of arithmetic progressions of length $3\leq k<110.$ Finally, Bennett and Siksek in a recent paper \cite{BS} shoved that if $k$ is large enough, the equation in question has only finitely many solutions.

One can also note that equation (\ref{maineq}), in the case $T=(0,1,\ldots,n-1)\in A_{n}$, was studied in a recent paper of Hajdu, Laishram and the first author \cite{HLT}. They proved that for $n\geq 1$ and $m\geq 2$ (with $n\neq 2$ in case of $m=2$) equation (\ref{maineq}) has only finitely many integer solutions. Moreover, they were also able to solve the equation explicitly for $n\leq 10$.

Equation (\ref{maineq}) can be considered as a generalization of the Erd\H{o}s and Selfridge equation (\ref{erse}).
Moreover, the problem of Erd\H{o}s and Graham \cite{ErGr} asking for the integer solution of the equation is given by
\begin{equation}\label{ergr}
y^m=P_{T}(x_{1},\ldots,x_{n}),
\end{equation}
where $T\in B_{n}, P_{T}(x_{1},\ldots,x_{n})=p_{a_{1}}(x_{1})\cdot\ldots\cdot p_{a_{n}}(x_{n})$ and
$$
B_{n}=\{(a_{1},\ldots,a_{n})\in\N^{n}:\;a_{i}\leq a_{i+1}\;\mbox{for}\;i=1,2,\ldots n-1\}.
$$
Here we impose the natural condition for solutions: $x_{i}+a_{i}<x_{i+1}$ for $i=1,\ldots,k-1$. This equation can be seen as a multi-variable (and multiplicative) analogue of equation (\ref{erse}). In the literature there are many nice results dealing with special cases of the problem of Erd\H{o}s and Graham and its various generalizations. In order to get more information on this problem one can consult the papers \cite{BaBe, BvL, LuWa, Skalba, Ulas2005}. Let us observe that if $k\geq 2$ then the above equation defines an algebraic variety of degree $\op{max}\{2,a_{1}+\ldots+a_{k}\}$ and dimension $k$. Based on existing results one can expect that if the number $k-\op{max}\{a_{1},\ldots,a_{k}\}$ is sufficiently large, then equation (\ref{ergr}) has infinitely many solutions in integers (however, this is in general an unproven conjecture).

Motivated by research devoted to the study of equation (\ref{ergr}) it is quite natural to consider a multi-variable and  additive version of (\ref{erse}) as well. More precisely, we are interested in the problem of existence of integer solution of the Diophantine equation
\begin{equation}\label{nequation}
z^m=G_{T}(x_{1},\ldots,x_{n}),
\end{equation}
where for a given $T=(a_{1},\ldots,a_{n})\in\ B_{n}$ we put
$$
G_{T}(x_{1},\ldots,x_{n})=\sum_{i=1}^{n}p_{a_{i}}(x_{i}).
$$
In some sense this equation is a bit simpler then equation (\ref{ergr}) because the degree of the underlying algebraic variety is governed by the number $a_{n}$. Indeed, in the sequel we will see that we can prove the existence of infinitely many integer solutions of equation (\ref{nequation}) for some values of $T\in B_{n}$ such that the question corresponding the existence of integer solutions of equation (\ref{ergr}) is open.

Let us describe the content of the paper in some details.

In Section \ref{sec2} we present some general finiteness results concerning the integer solutions of the Diophantine equation (\ref{maineq}). In particular, if $n\geq 2, T\in A_{n}$ with $a_{1}\geq 2$, then our equation has only finitely many solutions in integers.

In Section \ref{sec3} we consider equation (\ref{maineq}) with $m=2$ and $T\in A_{n}, n\leq 5$. With our assumption the genus of the corresponding curve $C_{T}:\;y^2=g_{T}(x)$ is bounded by 2. We briefly explain the available tools to determine integral and rational points on these curves and we provide details in a supplementary file that can be downloaded from \cite{shrek}. 

In Section \ref{sec4} we discuss the possible application of Runge's method and provide details in a case where we applied a kind of reduction of bound procedure. In the earlier mentioned supplementary file we provide results in case of $m=2$ and $n=5, 7, 9, 11, 13, 15,$ here we give the full description of the set of integer solutions of (\ref{maineq}) for each $T\in A_{n}.$ Let us note that the number of equations we solved by Runge's method is more then 40000. Moreover, we also apply Runge's method in case of equation (\ref{maineq}) with $(m,n)\in\{(8,3),(11,3),(4,5),(9,5),(6,7)\}$ and each $T\in A_{n}.$


Finally, Section \ref{sec5} is devoted to a multi-variable generalization of equation (\ref{maineq}). In particular, if $n\in\N_{\geq 2}$ we prove that equation (\ref{nequation}) has infinitely many integer solutions in certain cases.

Through the paper we also formulate several questions and conjectures concerning various Diophantine equations which may stimulate further research.

\section{The equation $y^m=g_{T}(x)$ for $T\in A_{n}$, with $n\geq 3$}\label{sec2}

In this section we consider equation (\ref{maineq}) for $T\in A_{n}, n\geq 3$ with some additional constraints on the shape of the sequence $T$. However, before we state our findings we will recall some general results concerning the solvability in integers $x, y, m$ of Diophantine equations of the form
\begin{equation}\label{genres}
y^{m}=g(x),
\end{equation}
where $g\in\Z[x]$ is fixed of degree $d$ and height $H$, where by height of the polynomial $g$ we understand the maximum of the modulus of the coefficients.

The following lemma, due to Tijdeman \cite{Tij}, will be one of our main tools.
\begin{lem}\label{L1}
If $g(x)$ has at least two distinct roots and $|y|>1$, then in the Diophantine equation {\rm (\ref{genres})} we have $m<c_{1}(d,H)$, where $c_{1}(d,H)$ is an effectively computable constant depending only on $d$ and $H$.
\end{lem}

The next result is a special case of a theorem of Brindza \cite{Bri}.

\begin{lem}\label{L2}
Suppose that one of the following conditions holds:
\begin{enumerate}
\item $m\geq 3$ and $g(x)$ has at least two roots with multiplicities co-prime to $m$,
\item $m=2$ and $g(x)$ has at least three roots with odd multiplicities.
\end{enumerate}
Then all integer solutions of equation {\rm (\ref{genres})} satisfies $\op{max}\{|x|,|y|\}\leq c_{2}(d, H)$, where $c_{2}(d,H)$ is an effectively computable constant depending only on $d$ and $H$.
\end{lem}

In order to apply the above results to our Diophantine equation (\ref{maineq}) we collect basic properties of the sequence of polynomials $(g_{T})_{T\in A_{n}}$ in the following. We note that in \cite{BBHL} the authors provided effective finiteness result for the equation $g_T(x)=ay^m+b$ in case of $T=(0,1,\ldots,n-1).$

\begin{lem}\label{firstprop}
Let $n\in\N_{\geq 2}, T=(a_{1},\ldots,a_{k})\in A_{n}$ and $a_{1}\geq 1$.
\begin{enumerate}
\item We have $g_{T}(x)=p_{a_{1}}(x)h_{T}(x)$, where $h_{T}\in\Z[x]$ and $\op{deg}h_{T}=n-(a_{1}+1)$.
\item The roots $x=-i, i=0,\ldots,a_{1}$ of the polynomial $g_{T}$ are simple. In particular, the polynomial $g_{T}(x)$ has at least two roots with odd multiplicity.
\item If $n\geq 5$ and $a_{1}=1, a_{2}=3, a_{3}\geq 5$, then the polynomial $h_{T}(x)$ is not a square of a polynomial with integer coefficients. In particular, the polynomial $g_{T}(x)$ has at least three roots with odd multiplicity.
\item The equation $g_{T}(x)=\pm 1$ has no solutions in integers.
\end{enumerate}
\end{lem}
\begin{proof} We have
$$
g_{T}(x)=p_{n}(x)+\sum_{i=1}^{k}p_{a_{i}}(x),
$$
where $1\leq a_{1}<a_{2}<\ldots<a_{k}<n$. In particular, $p_{a_{1}}(x)|p_{a_{i}}(x)$ for $i=1,\ldots, k$ and obviously $p_{a_{1}}(x)|p_{n}(x)$. We also have the general identity
$$
p_{a+b}(x)=p_{b}(x)p_{a-1}(x+b+1).
$$
Consequently, we obtain the following relation
$$
g_{T}(x)=p_{a_{1}}(x)(1+g_{T'}(x+a_{1}+1)),
$$
where $T'=(a_{2}-a_{1}-1,a_{3}-a_{1}-1,\ldots, a_{k}-a_{1}-1,n-a_{1}-1)$.
We thus have $h_{T}(x)=1+g_{T'}(x+a_{1}+1)$. Now, let us observe that for $x_{0}=0,\ldots,-a_{1}$ and any $a>a_{1}$ we have $p_{a}(x_{0}+a_{1}+1)> 0$. This implies that $h_{T}(x_{0})=1+g_{T'}(x_{0}+a_{1}+1)>0$ and thus the roots $0, -1, \ldots, -a_{1}$, of the polynomial $g_{T}(x)$ are all simple.
Consequently, under our assumptions on $n$ and $T$, the polynomial $g_{T}(x)$ has at least two roots with multiplicity equal to one.

If $a_{1}=1, a_{2}=3$ and $a_{3}\geq 5$ then we get that
\begin{align*}
h_{T}(0) &=\lim_{x\rightarrow 0}\frac{g_{T}(x)}{x(x+1)}=1+g_{T'}(2)\equiv 1+2\cdot 3\equiv 3\pmod{4},\\
h_{T}(-1)&=\lim_{x\rightarrow -1}\frac{g_{T}(x)}{x(x+1)}=1+g_{T'}(1)\equiv 1+1\cdot 2\equiv 3\pmod{4}.
\end{align*}
In particular, the polynomial $h_{T}(x)$ cannot be a square of a polynomial with integer coefficients and thus has at least one root of odd multiplicity. Consequently, the polynomial $g_{T}(x)$ has at least three roots of odd multiplicity.




Let us observe that if $g_{T}(x)=\pm 1$ for some $x\in\Z$, then necessarily $x\cdot\ldots\cdot(x+a_{1})=\pm 1$, which is clearly impossible for $a_{1}\geq 1$.

\end{proof}

\begin{rem}\label{reminfi}
{\rm We note that in general we cannot have similar result in the case $a_{1}=0$ (but see Conjecture \ref{conj1} below). Indeed, if $n=3$ and $T=(0,1)$ then
$$
g_{T}(x)=x(x+2)^3.
$$
Consequently, the equation $g_{T}(x)=y^3$ has infinitely  many integer solutions of the form $(x,y)=(t^3,t(t^3+2))$, where $t\in\Z$. Moreover, let us note that in this case the equation $g_{T}(x)=y^{4}$ has infinitely many rational solutions. Indeed, let us take $y=t(x+2)$. We get the equation $t^4(x+2)^4=x(x+2)^3$. Consequently, by solving for $x$ we see that for each $t\in\Q\setminus\{-1,1\}$ the pair
$$
(x,y)=\left(\frac{2t^4}{1-t^4},\frac{2t}{1-t^4}\right)
$$
is a solution of our equation.
}
\end{rem}

\begin{rem}
{\rm Let us observe that the above lemma can be further generalized. Indeed, instead of working with the polynomial $p_{n}(x)$ one can consider a more general form. Let us consider the polynomial
$$
P_{p,q,n}(x)=\prod_{i=0}^{n}(px+q),
$$
where $p\in\N, q\in\Z$ and $|q|<p$. Then, for $T=(a_{1},\ldots,a_{k})\in A_{n}, n\in\N$, a similar lemma can be proved for the polynomial
$$
G_{p,q,T}(x)=P_{p,q,n}(x)+\sum_{i=1}^{k}P_{p,q,a_{i}}(x).
$$
}
\end{rem}
As an immediate consequence of the above lemmas we get the following result.

\begin{thm} Let $n\in\N_{\geq 2}, T=(a_{1},\ldots,a_{k})\in A_{n}$. If $a_{1}\geq 2$ or $a_{1}=1, a_{2}=3, a_{3}\geq 5$ then for the integer solutions of the Diophantine equation $y^m=g_{T}(x)$ we have:
\begin{enumerate}
\item if $y\neq 0$, then $m<c_{1}(n)$,
\item if $m\geq 3$, then $\op{max}\{m,|x|,|y|\}<c_{2}(n)$,
\item if $m=2$, then $\op{max}\{|x|,|y|\}<c_{3}(n)$.
\end{enumerate}
Here $c_{1}(n), c_{2}(n), c_{3}(n)$ are effectively computable constants depending only on $n$.
\end{thm}
\begin{proof}
The first part is an immediate consequence of Lemma \ref{firstprop} and Lemma \ref{L1}.

In order to get the second part we note that the roots $x=0,-1$ of the polynomial $g_{T}(x)$ are simple. Moreover, the degree $\op{deg}g_{T}(x)=n+1$ is grater then two, and our claims follow from the second part of Lemma \ref{L2}.

Finally, in order to get the last part from the statement we note that the roots $x=0,-1,-2$ of the polynomial $g_{T}(x)$ are simple (in case of $a_{1}\geq 2$) or that the roots $x=0,-1$ are simple and we have one more root with odd multiplicity of the polynomial $h_{T}(x)$ (which is a consequence of the third part of Lemma \ref{firstprop}). Moreover, the degree $\op{deg}g_{T}(x)=n+1$ is grater then two, and our claims follow from the second part of Lemma \ref{L2}.
\end{proof}


\begin{rem}
{\rm The crucial property which guarantees the finiteness of the set of integer solutions of equation (\ref{maineq}) is the number of multiple roots of the polynomial $g_{T}$, with $T\in A_{n}$. Based on our computations we formulate the following conjecture.

\begin{conj}\label{conj1}
Let $n\in\N$ and $T\in A_{n}$ be given. The polynomial $g_{T}(x)$ has multiple roots if and only if:
\begin{enumerate}
\item $T=(n-4)$ for $n\geq 4$ with $(x^2+(2n-3)x+n^2-3n+1)^2| g_{T}(x)$,
\item $T=(n-3,n-2)$ for $n\geq 3$ with $(x+n-1)^3|g_{T}(x)$,
\item $T=(n-2,n-1)$ for $n\geq 2$ with $(x+n)^2|g_{T}(x)$.
\end{enumerate}
In each of the above cases the corresponding co-factors have no multiple roots.
\end{conj}
}
\end{rem}

\section{Rational solutions of the equation $y^2=g_{T}(x)$ with $T\in A_{n}, n\leq 5$}\label{sec3}

Let $n\in\N$ and for given $T\in A_{n}$ let us consider the algebraic curve
$
C_{T}:\;y^2=g_{T}(x).
$
Let us write $gen(T):=\op{genus}(C_{T})$ - the genus of the curve $C_{T}$ and $J_{T}:=\op{Jac}(C_{T})$ - the Jacobian variety associated with $C_{T}$. Moreover, we define $r(T):=\op{rank}(J_{T})$ - the rank of the Jacobian variety $J_{T}$. As usual, by $C_{T}(\Q)$ we will denote the set of all rational points on the curve $C_{T}$ and by $C_{T}(\Z)$ - the set of integral points on $C_{T}$.


If $T\in A_2,A_3$ or $A_4,$ then using standard method of parametrization of curves we get the description of the set of rational points in cases such that $gen(T)=0,$ e.g.
for $T=(0)$ we have $g_{T}(x)=x(x^2+5x+5)^2$ and the description of the set $C_{T}(\Q)$ is
$\{(t^2,t(t^4+5t^2+5)):\;t\in\Q\}.$ If $T\in A_2,A_3$ or $A_4$ and $gen(T)=1,$ then we get the list of integral points (rational points in case of rank 0 curves) by MAGMA \cite{Mag}. If $T\in A_4,$ then we also obtain genus 2 curves, fortunately, in each case, the rank of the Jacobian variety $J_{T}$ associated with $C_{T}$ is bounded by 1. Thus in each case we can apply Chabauty's method \cite{Chab} in order to find complete set of rational points on the curve $C_{T}$. The procedures in case of genus 2 curves were implemented in {\sc Magma} based on papers by Stoll \cite{StollB1,Stoll,StollB2}. More precisely, in case of $r_{T}=0$ we can use directly the following commands:

{\tt
\hskip 0.5cm
\begin{tabular}{l}
A<x>:=PolynomialRing(Rationals());\\
C:=HyperellipticCurve(f(x));\\
J:=Jacobian(C);\\
Chabauty0(C);
\end{tabular},
}

\noindent where $f$ is our polynomial of degree 5 or 6 without multiple roots and such that rank of $J$ is equal to 0. The computation of bound for the rank is performed with the procedure {\tt RankBound(J)}.

In the case when the rank of the Jacobian is equal to 1, the situation is a bit different. We first apply the procedure {\tt Points(J: Bound:=$10^3$)} in order to find the set, say $A$, of all rational divisors on the Jacobian $J$ with the height bounded by $10^3$. The chosen bound is grater then $e^{h(J)}$, where $h(J)$ is the height constant associated with the Jacobian under consideration. Then, we compute the reduced basis, say $R$, with the help of the procedure {\tt ReducedBasis(A)}. Next, we look for divisors of infinite order in $R$, by checking the order of the elements of $R$ with the help of procedure {\tt Order(P)} for each $P\in R$. Finally, for the set $B$ of all divisors of infinite order we apply the procedure {\tt Chabauty(B)} and get the required set of all rational points on $C$. For example, if $n=5$ and $T=(4)$, then
$$
C_{T}:\;y^2=x(1+x)(2+x)(3+x)(4+x)(6+x).
$$
In particular, $gen(T)=2, r_{T}=1$ and the set of finite rational points on the curve $C_{T}$ is as follows
$$
C_{T}(\Q)=\{(-6,0), (-4,0), (-3,0), (-2,0), (-1,0), (0,0), (-12/7, \pm 720/7)\}.
$$
Similarly, if $T=(2,3,4)$ then
$$
C_{T}:\;y^2=x(1+x)(2+x)(4+x)(19+9x+x^2)
$$
and $gen(T)=2, r_{T}=1$. Then we get
$$
C_{T}(\Q)=\{(-38/11, \pm 1368/11), (-4,0), (-2,0), (0,0)\}.
$$
Detailed results can be downloaded from \cite{shrek}. Unfortunately, we were unable to characterize rational solutions on the curve $C_{T}$ for some $T\in A_{5}$. However, numerical computations allow to formulate the following

\begin{conj}
	Let $T\in A_{5}$ and $r(T)>1$, then the corresponding values of $T$ and the sets $C_{T}(\Q)$ are as follows.
	
	In the description of the set $C_{T}(\Q)$ we omit the points at infinity $(1,\pm 1, 0)$.
	\begin{center}
		\begin{equation*}
		\begin{array}{|l|l|l|}
		\hline
		T         & r(T) & C_{T}(\Q) \\
		\hline
		\hline
		(2)       & 2 & \{(-2,0), (-1,0), (0,0), (2/3, \pm 460/3)\} \\
		\hline
		(3)       & 2 & \{(-9, \pm 252), (-3,0), (-2,0), (-1,0), (0,0), (-1,0), (0,0), (-18/5, \pm 468/5)\} \\
		\hline
		(0,4)     & 4 & \{(0,0), (1,\pm 29), (9/4, \pm 5871/4)\} \\
		\hline
		(1,3)     & 2 & \{(-4,\pm 6), (-1,0), (0,0)\} \\
		\hline
		(1,4)     & 2 & \{(-1,0), (0,0)\} \\
		\hline
		(0, 1,2)     & 2 & \{(-2,0), (0,0), (1,\pm 27)\} \\
		\hline
		(0,2,3)     & 2 & \{(0,0)\} \\
		\hline
		(1,2,3)   & 2 & \{(-3,0), (-1,0), (0,0)\} \\
		\hline
		(1,3,4)   & 2 & \{(-4,\pm 6), (-1,0), (0,0), (-25/9,\pm 620/9)\} \\
		\hline
		(0,1,2,3)   & 3 & \{(-4,0), (-2,0), (-1,0), (0,0), (-13/3, \pm 91/3), (-5/3, \pm 55/3)\} \\
		\hline
		(0,1,3,4)   & 2 & \{(-3,0), (-1,0), (0,0)\} \\
		\hline
		\end{array}
		\end{equation*}
	\end{center}
	\begin{center} {\rm Table 1. Conjectured form of the set of all (finite) rational points on the genus two curves $C_{T}$, where $T\in A_{5}$ and $r(T)\geq 2$.} \end{center}
	
\end{conj}
We tried to determine the complete list of rational points on the remaining genus 2 curves via the so-called elliptic Chabauty's method \cite{NB1,NB2}, but there were cases we could not find sufficiently many independent points or we could not compute the rank of the genus 1 curve defined over a number field. The most difficult one seems to be the case with $T=(0,4).$ The genus 2 curve is given by
$
y^2=x(x^5 + 16x^4 + 95x^3 + 260x^2 + 324x + 145).
$
The rank of the Jacobian is 4 and one needs to work over a degree 5 number field.

\section{Application of Runge method for several equations $y^{m}=g_{T}(x)$}\label{sec4}

Consider the Diophantine equations $y^2=g_{T}(x)$ for $T\in A_5,A_7,A_9,A_{11}$ and $A_{13}.$
In all cases $g_T(x)$ is a monic polynomial of degree 6, hence Runge's condition is satisfied. An algorithm to solve such Diophantine equations is given in \cite{TenFG}, we followed it to determine the integral solutions. We note that in case of $T\in A_{13}$ there are $2^{13}-1$ equations to be solved and the bounds obtained by Runge's method are of size $10^6.$ Therefore we applied a modified version of the reduction argument used in \cite{TenFG}. We illustrate the idea through an example. If $T=(2, 3, 4, 5, 7, 9, 10, 12),$ then the equation is given by
{\small
	\begin{eqnarray*}
	y^2&=&(x^{10} + 85x^9 + 3200x^8 + 70211x^7 + 993342x^6 + 9458533x^5 + 61303921x^4 + 266606990x^3 + \\
	&&742982499x^2 + 1194792102x + 838752409)(x + 4)(x + 2)(x + 1)x
	\end{eqnarray*}}
The polynomial part of the Puiseux expansion of $g_T(x)^{1/2}$ is given by
$$
P_T(x)=x^{7} + 46 \, x^{6} + \frac{1693}{2} \, x^{5} + \frac{15931}{2} \, x^{4} + \frac{323643}{8} \, x^{3} + \frac{212995}{2} \, x^{2} + \frac{1953743}{16} \, x + \frac{574129}{16}.
$$
We obtain that
\begin{eqnarray*}
256g_T(x)-(16P_T(x)-1)^2&&\mbox{ has roots in the interval } I_a:=[-68,\ldots,2.018\times 10^6],\\
256g_T(x)-(16P_T(x)+1)^2&&\mbox{ has roots in the interval } I_b:=[-1.01\times 10^6,\ldots,0].
\end{eqnarray*}
Hence it remains to solve the equations
\begin{eqnarray*}
y^2&=&g_T(x) \mbox{ with }x\in[-1.01\times 10^6,\ldots,2.018\times 10^6],\\
P_T(x)^2-g_T(x)&=&0.
\end{eqnarray*}
Therefore the total number of equations to handle is 3026952.
We compute the appropriate intervals in case of two positive integers $k_1,k_2:$
\begin{eqnarray*}
256g_T(x)-(16P_T(x)-k_1)^2&&\mbox{ has roots in the interval } I_1,\\
256g_T(x)-(16P_T(x)+k_2)^2&&\mbox{ has roots in the interval } I_2.
\end{eqnarray*}
For some fixed $k_1,k_2$ we determine the integral solutions of the equations
\begin{eqnarray*}
y^2&=&g_T(x) \mbox{ with }x\in I_1\cup I_2,\\
(P_T(x)+k/16)^2-g_T(x)&=&0 \mbox{ for some values of $k$ depending on $k_1,k_2.$}
\end{eqnarray*}
The goal is to reduce the number of these type of equations. Based on numerical experiences we start with $k_1=|I_a|^{1/4},k_2=|I_b|^{1/4},$ where $|\cdot|$ denotes the number of integers in the given interval.
We compute the intervals $I_1,I_2$ for these values of $k_1,k_2.$ If the number of equations is smaller than in the previous step, then $k_1=2|I_a|^{1/4}$ and $k_2=2|I_b|^{1/4}.$ Therefore at the end we will have $k_1=i_1|I_a|^{1/4}$
and $k_1=i_2|I_b|^{1/4}$ for some integers $i_1,i_2.$ In the above example the interval $I_a$ is reduced to $[-69,\ldots,2280]$ in 23 steps, so $k_1=851$, the interval $I_b$ is reduced to $[-1674,\ldots,0]$ in 20 steps, thus $k_2=620.$
The number of equations to handle before reduction was more than 3 million, after the above reduction it is less than 6000.

We also note that equations for which $\gcd(m,n+1)\geq 2$ can be solved using Runge's method. For example if $n=14$ and $T=(10,11,12,13),$ then we have
$$
y^m={\left(x^{3} + 39 \, x^{2} + 504 \, x + 2157\right)} {\left(x + 12\right)}p_{10}(x),
$$
an equation that can be solved using Runge's method for $m=3,5$ and 15. We note that in all cases only the trivial solutions with $y=0$ exist. In this way we were able to determine all solutions of equation \eqref{maineq} with $(m,n)\in\{(5,3),(8,3),(11,3),(4,5),(9,5),(6,7)\}.$
Detailed results in case of $(m,n)=(5,3)$ can be downloaded from \cite{shrek}.
There are only a few types of non-trivial solutions, these are as follows
$$
\{(-8,-2),(-8,12),(-5,-5),(-4,2),(-1,-1)\}.
$$
According to our computations we see that there are very few solutions of the equation $y^{m}=g_{T}(x)$ with $xy\neq 0$. This may suggest to treat the equation $y^{m}=g_{T}(x)$ as an equation in an unbounded number of variables and look for its solutions in $m\in\N_{\geq 2}, n\in\N, T=(a_{1},\ldots,a_{k})\in A_{n}$ and $x, y\in\Z$ satisfying the condition $xy\neq 0$. However, stating the problem in this way one can easily get infinitely many solutions. Indeed, in Remark \ref{reminfi} we observed that the equation $y^{3}=x(x+2)^{3}=g_{T}(x)$, with $n=3$ and $T=(0,1)$ has infinitely many solutions of the form $(x,y)=(u^3,(u(u^3+2))$. Thus, by taking $u<0$, we see that for each $n\in\N_{\geq 3}$ and $T'\in A_{n}$ of the form $T'=(0,1,a_{3},\ldots, a_{k})$, where $a_{3}\geq |u|$ we have $g_{T'}(u)=(u(u^3+2))^3$. This is a consequence of the vanishing of $p_{a_{i}}(u)$ for $a_{i}\geq |u|$. Let us also note that essentially each negative value of $x$ which is a solution of the equation $y^m=g_{T}(x)$ for some $m\in\N$ and $T\in A_{n}$ leads in the same way to infinitely many of $T'$ such that $y^{m}=g_{T'}(x)$. These observations motivate us to state the following

\begin{ques}\label{prob1}
For which integer values of $x<0$ is there an $m\in\N_{\geq 2}, n\in\N_{\geq 2}$ and $T\in A_{n}$ such that $y^{m}=g_{T}(x)$ with $y\neq 0$?
\end{ques}

Let us introduce the set
$$
\cal{N}:=\{x\in\N_{\leq 0}:\;\mbox{there is}\; (m, n, T)\in \N_{\geq 2} \times \N_{\geq 2} \times A_{n}:\; y^{m}=g_{T}(x)\;\mbox{for some}\;y>0\}.
$$
We saw that $-u^3\in\cal{N}$ for $u\in\N$. Let us also note that $-9,-5,-4\in\cal{N}$. Indeed, the pair $(x,y)=(-9,252)$ is a solution of the equation $y^2=g_{T}(x)$ for $n=5, T=(3)$ and
$(x,y)=(-5,-5)$ is a solution of the equation $y^3=g_{T}(x)$ for $n=4, T=(0)$. Moreover, $(x,y)=(-4,6)$ is a solution of the equation $y^2=g_{T}(x)$ for $n=3, T=(1)$. We do not know any other elements of $\cal{N}$. However, if we allow $x$ to be {\it rational}, then we get some additional interesting solutions. Indeed, one can easily prove the identities:
$$
p_{4k+3}\left(-\frac{4k-1}{2}\right)+p_{4k-1}\left(-\frac{4k-1}{2}\right)=\left(\frac{16k^2+32k+11}{4^{k+1}}\prod_{i=0}^{2k-1}(2i+1)\right)^2,
$$
$$
p_{4k+1}\left(-\frac{4k-1}{2}\right)+p_{4k}\left(-\frac{4k-1}{2}\right)+p_{4k-1}\left(-\frac{4k-1}{2}\right)=\left(\frac{4k+3}{2^{2k+1}}\prod_{i=0}^{2k-1}(2i+1)\right)^{2}
$$
In order to prove the first equality we note $p_{4k+3}(x)+p_{4k-1}(x)=p_{4k-1}(x)(x^2+(8k+3)x+16k^2+12k+1)^2$. Moreover,
\begin{align*}
p_{4k-1}\left(-\frac{4k-1}{2}\right)&=\prod_{i=0}^{2k-1}\left(-\frac{4k-1}{2}+i\right)\left(-\frac{4k-1}{2}+i+2k\right)\\
                                    &=\frac{1}{4^{2k}}\prod_{i=0}^{2k-1}(4k-2i-1)(2i+1)=\frac{1}{4^{2k}}\prod_{i=0}^{2k-1}(2i+1)^2,
\end{align*}
where last equality follows from the equality of sets $\{2i+1:\;i\in\{0,2k-1\}\}=\{4k-2i-1:\;i\in\{0,2k-1\}\}$.

In order to prove the second equality it is enough to note the identity $p_{4k+1}(x)+p_{4k}(x)+p_{4k-1}(x)=p_{4k-1}(x)(x+1)^2$.

In the light of the above identities one can formulate the following

\begin{ques}\label{ques2}
Let us consider the equation $y^{m}=g_{T}(x)$ in unbounded many variables $m\in\N_{\geq 2}, n\in\N_{\geq 2}, T\in A_{n}$, where $m, n$ are chosen in such a way that the genus of the curve defined by our equation is positive. Is the set of positive integer solutions infinite?
\end{ques}

We known only very few solutions of the equation from the question above. For the convenience of the reader we collect them in the table below:
\begin{center}
\begin{equation*}
\begin{array}{|l|l|}
\hline
  x & [m,\;n,\;T] \\
  \hline
  1 & [2, 4, (0)], [2, 5, (0,4)], [2, 5, (0,1,2)], [2, 6, (0)], [2, 6, (3,4)], [2, 7, (0,3,4,5,6)],  \\
    & [2, 8, (0,3,7)], [2, 8, (0,1,2,5)], [2, 9, (0,1,2,5,6,7], [2, 14, (0,1,2,6,\ldots,13)]\\
    & [3, 5, (0,1,2)], [5, 3, (1,2)], [7, 4, (1,2)]\\
  2 & [2, 3, (2)], [2, 5, (2,3)], [7, 3, (0,1)] \\
  4 & [2, 6, (0,4)]\\
  \hline
\end{array}
\end{equation*}
\end{center}
\begin{center}
Table 2. Known solutions of the equation $y^{m}=g_{T}(x)$ with positive values of $x$.
\end{center}

One can also ask about positive {\it rational} solutions but we were unable to prove anything similar like in the case of negative solutions.

Let us also observe that the problem concerning the existence of solutions with $x=1$ is equivalent with the problem of finding integer solutions of the Diophantine equation of polynomial-factorial type
$$
y^{m}=\sum_{i=1}^{n}(a_{i}+1)!
$$
in non-negative integers $a_{1}, a_{2}, \ldots$ and $y, m\in\N$. This is a consequence of the equality $p_{a}(1)=(a+1)!$.

\section{Some results concerning an additive version of Erd\H{o}s-Graham question}\label{sec5}

In this section we study the problem of existence of integer solutions of the Diophantine equation (\ref{nequation}). We are mainly interested in the case when $m=2, 3$ and consider the related Diophantine equations for certain sequences chosen from the set $B_{2}$.

From the geometric point of view equation (\ref{nequation}) defines an algebraic variety of dimension $n$ and degree $\op{max}\{m,a_{n}+n\}$. As usual, the general expectation (when dealing with Diophantine equations with small degree and many variables) is the following. If $m$ is not too large and $\op{max}\{a_{1},\ldots,a_{m}\}$ is relatively small compared with $m$, then equation (\ref{nequation}) should have infinitely many solutions in integers.
\begin{rem}
{\rm Let us recall that if $a_{1}=2$ or $a_{1}=a_{2}=3$, then for each $n-2$ tuple $(a_{3},\ldots,a_{n})\in B_{n-2}$, the Diophantine equation
$$
y^2=p_{a_{1}}(x_{1})p_{a_{2}}(x_{2})\ldots p_{a_{n}}(x_{n})
$$
has infinitely many solutions in positive integers $(x_{1},\ldots, x_{n})$ satisfying the condition $x_{i}+a_{i}<x_{i+1}$ for $i=1,\ldots,n-1$. This was proved by Bauer and Bennett in \cite{BaBe}.
An (additive) analogue of the above result of Bauer and Bennett can be obtained via the identities
$$
p_{1}(x-1)+x=x^2, \quad p_{2}(x-1)+x=x^3
$$
with $x=\sum_{i=2}^{n}p_{a_{i}}(x_{i}).$
}
\end{rem}
By applying results from the theory of Pellian equations combiningcombined with certain polynomial identities we dealt with certain equations of the form $z^m=p_i(x)+p_i(y).$ We proved that there are infinitely many solutions $(x,y,z)$ in integers (polynomials). For example, if $m\equiv 1\pmod{2}$ then the Diophantine equation
$$
z^{m}=p_{1}(x)+p_{1}(y)
$$
has a polynomial solution
$$ x=2^{\frac{n-1}{2}}t^{m}-1,\quad y=2^{\frac{n-1}{2}}t^{m},\quad z=2t^2. $$
Other results can be found in the supplementary material at \cite{shrek}.

Now we concentrate on the case $m=3$ with $(n,a_{1},a_{2})=(2,2,2)$, i.e., we consider the Diophantine equation
\begin{equation}\label{eq333}
z^{3}=p_{2}(x)+p_{2}(y).
\end{equation}
From the general observation given on the beginning of this section we have the solution $x=t, y=t(t+1)(t+2)-1$, where $t$ is an integer parameter. We thus are interested in different solutions ($x,y,z)$ of equation (\ref{eq333}), i.e., do not satisfying the relation $y=x(x+1)(x+2)-1$. However, before we present our result we will need a well known property of Pell type equations. More precisely, if $(X,Z)=(X',Z')$ is a particular solution of the Diophantine equation $X^2-AZ^2=B$ and $(X,Z)=(X'',Z'')$ with $Z''\neq0$, is a solution of the equation $X^2-AZ^2=1$, then for each $n\in\N$, the pair $(X,Z)=(X_{n},Z_{n})$, defined recursively by $X_{0}=X', Z_{0}=Z'$ and for $n\geq 1$ by
\begin{equation}\label{solpell}
X_{n}=X''\cdot X_{n-1}+AZ''\cdot Z_{n-1},\quad Z_{n}=Z''\cdot X_{n-1}+X''\cdot Z_{n-1},
\end{equation}
is the solution of the equation $X^2-AZ^2=B$.

\begin{thm}\label{sol333}
	The Diophantine equation {\rm (\ref{eq333})} has infinitely many solutions $(x,y,z)$ in polynomials with integer coefficients and satisfying $\op{deg}_{t}x=\op{deg}_{t}y$.
\end{thm}
\begin{proof}
	The factorization $p_{2}(x)+p_{2}(y)=(x+y+2)(x^2-xy+y^2+x+y)$ suggests a reasonable assumption that there are solutions of (\ref{eq333}) satisfying the divisibility condition $(x+y+2)|z$. After some numerical experiments we observed that the quotient $z/(3(x+y+2))$ is square of an integer. We thus write $z=3t^2(x+y+2)$, where $t$ is a variable taking integer values. We cancel the common factor $x+y+2$ and are left with the equation of Pell type
	\begin{equation}\label{Pell1}
	U^2-3(108t^6-1)V^2=12(2916t^6-135t^6+1),
	\end{equation} where $$ U=3(108t^6-1)(x+1),\quad V=(54t^6+1)x+2(27t^6-1)y+108t^6-1 $$ or equivalently $$ x=\frac{U}{3(108t^6-1)}-1,\quad y=\frac{3(108t^6-1)V-(54t^6+1)U}{6(27t^6-1)(108t^6-1)}-1. $$ In other words, in order to construct polynomial solutions of equation (\ref{eq333}) it is enough to prove that there are infinitely many polynomial solutions $(U,V)$ of equation (\ref{Pell1}) satisfying the congruence relations \begin{align*}
	U                      &\equiv 0\pmod{3(108t^6-1)},\\
	3(108t^6-1)V-(54t^6+1)U&\equiv 0\pmod{6(27t^6-1)(108t^6-1)}.
	\end{align*}
	
	We observe that equation (\ref{Pell1}) has a solution
	$$
	U'=3(6t^3+1)(108t^6-1),\quad V'=108t^6+18t^3-1.
	$$
	Moreover, the equation
	$$
	U^2-3(108t^6-1)V^2=1
	$$
	has the solution
	$$
	U''=(6t^2-1)(36t^4+6t^2+1),\quad V''=12t^3.
	$$
	Consequently, following the remark given before the statement of our theorem, we see that for each $n\in\N_0$ the pair $(U_{n}, V_{n})$ of polynomials, where $U_{0}=U', V_{0}=V'$ and \begin{align*}
	U_{n}&=(6t^2-1)(36t^4+6t^2+1)U_{n-1}+36(108t^6-1)t^3V_{n-1},\\
	V_{n}&=12t^3U_{n-1}+(6t^2-1)(36t^4+6t^2+1)V_{n-1},
	\end{align*}
	for $n\geq 1$, is the solution of equation (\ref{Pell1}). First of all, we note that the leading coefficients of the polynomials $U_{n}$ and $V_{n}$ are positive. Consequently, by induction on $n$, we easily get the equalities:
	$$
	\op{deg}U_{n}(t)=3(2n+3),\quad \op{deg}V_{n}(t)=6(n+1).
	$$
	Moreover, having the values of degrees of our polynomials we can easily compute the leading coefficients:
	$$
	\op{LC}(U_{n}(t))=\frac{9}{2}432^{n+1},\quad \op{LC}(V_{n}(t))=\frac{1}{4}432^{n+1}.
	$$
	Next, we observe that $U_{0}=3(6t^3+1)(108t^6-1)\equiv 0\pmod{3(108t^6-1)}$ and from the recurrence relation for $U_{n}$ we get that for $n\geq 1$ the following congruence holds
	$$
	U_{n}\equiv U_{n-1}\pmod{3(108t^6-1)}.
	$$
	Thus, by induction on $n$ we immediately get that $U_{n}\equiv 0\pmod{3(108t^6-1)}$ for each $n\in\N$.
	
	The proof that for each $n\in\N_0$ we have $3(108t^6-1)V_{n}-(54t^6+1)U_{n}\equiv 0\pmod{6(27t^6-1)(108t^6-1)}$ is more complicated. First of all we note that \begin{align*} 3(108t^6-1)V_{0}-(54t^6+1)U_{0}&\equiv 6(1 - 6 t^3)(27t^6-1)(108t^6-1)\\
	&\equiv 0\pmod{6(27t^6-1)(108t^6-1)},
	\end{align*} and thus the congruence we are interested in is satisfied for $n=0$. In order to prove that the same is true for $n\in\N$, it is enough to prove the following congruences:
	\begin{align*}
	U_{n}&\equiv (108t^6-1)\left(\frac{3}{4}(7A_{n}-A_{n-1})t^3+A_{n}\right)\pmod{\lambda(t)},\\
	V_{n}&\equiv 9(63A_{n} - 9A_{n - 1} - 72)t^9+108A_{n}t^6-\frac{3}{4}(7A_{n}-A_{n-1}-32)t^3-A_{n}\pmod{\lambda(t)},
	\end{align*}
	where $\lambda(t)=6(27t^6-1)(108t^6-1), A_{0}=1, A_{1}=15$ and for $n\geq 2$ we have $A_{n}=14A_{n-1}-A_{n-2}$. We omit the tiresome proof of the above congruences and the fact that $7A_{n}-A_{n-1}\equiv 0\pmod{4}$ for $n\in\N$. It uses only induction and the recurrence relations satisfied by the sequences $(U_{n})_{n\in\N_0}, (V_{n})_{n\in\N_0}$ and $(A_{n})_{n\in\N_0}$.
	
	Consequently, we get that for each $n\in\N_0$ the polynomials $x_{n}, y_{n}$ defined by
	$$
	x_{n}(t)=\frac{U_{n}(t)}{3(108t^6-1)}-1,\quad y_{n}(t)=\frac{3(108t^6-1)V_{n}(t)-(54t^6+1)U_{n}(t)}{6(27t^6-1)(108t^6-1)}-1,
	$$
	with $z_{n}(t)=3t^2(x_{n}(t)+y_{n}(t)+2)$ are the solutions of the Diophantine equation (\ref{eq333}). Our theorem is proved.
	
	For example, for $n=1, 2$ we get the following polynomial solutions of equation (\ref{eq333}):
	\begin{equation*}
	\begin{array}{lll}
	x_{1}=2(1296t^9+216t^6-9t^3-1),      &  & x_{2}=6t^3(186624 t^{12}+31104 t^9-2160 t^6-216 t^3+5), \\
	y_{1}=-2(1296 t^9-216 t^6-9 t^3+1),  &  & y_{2}=-6t^3(186624 t^{12}-31104 t^9-2160 t^6+216 t^3+5),\\
	z_{1}=6t^2(432 t^6-1)             ,  &  & z_{2}=6t^2(186624 t^{12}-1296 t^6+1)
	
	\end{array}
	\end{equation*}
\end{proof}

\begin{rem}
	{\rm Tracing back our construction of the polynomials $U_{n}, V_{n}$ from the proof of the above theorem one can easily prove
		that
		$$
		\op{deg}x_{n}(t)=\op{deg}y_{n}(t)=3(2n+1)
		$$
		for $n\in\N_0$ and the expressions for the leading coefficients are as follows
		$$
		\op{LC}(x_{n}(t))=-\op{LC}(y_{n}(t))=6\cdot 432^{n}.
		$$
		Moreover, one can prove (a rather unexpected) equality $y_{n}(t)=x_{n}(-t)$ for each $n\in\N$.
	}
\end{rem}

\begin{rem}
	{\rm The family of polynomial solutions of equation (\ref{eq333}) constructed in the proof of Theorem \ref{sol333}, has quite unexpected property: $x_{n}(t)=y_{n}(-t)$ for each $n\in\N_0$. Moreover, there are no $n\in\Z$ such that both $x_{n}(t)$ and $y_{n}(t)$ are positive. Consequently, an interesting question arises whether there are infinitely many solution of equation (\ref{eq333}) satisfying $y>x>0$ and $y\neq x(x+1)(x+2)-1$. In order to find such solutions we performed numerical search and found that in the range $0<x<y<10^5$ we have only 10 solutions satisfying the required conditions. The solutions are the following:
		\begin{align*}
		(x,y)=&(97, 277), (176, 551), (263, 1104), (495, 503), (1244, 2472), (3986, 31706), \\
		&(4505, 12781), (24047, 30599), (26642, 40684), (94743, 96255).
		\end{align*}
	}
\end{rem}

\begin{rem}
	{\rm In the range $1\leq x\leq y\leq 10^5$, equation $z^{2}=p_{2}(x)+p_{2}(y)$ has 619 integer solutions. This relatively large number suggests the existence of a polynomial solution. We tried quite hard to construct parametric solutions but we failed. This motivates us to formulate the following problem.
		
		\begin{ques}
			Does the equation $z^2=p_{2}(x)+p_{2}(y)$ has a solution in polynomials with integer coefficients?
		\end{ques}
	}
\end{rem}

\begin{thm}
Let $i\in\{3,4\}$. The equation $z^2=p_{i}(x)+p_{i}(y)$ has infinitely many solutions in positive integers.
\end{thm}
\begin{proof} In order to get the result we are looking for rational numbers $a, b$ such that the polynomial $F_{a,b,i}(x):=p_{i}(x)+p_{i}(ax+b)$ has multiple roots. The necessary and sufficient condition for this property is the vanishing of the discriminant of the polynomial $F_{a,b,i}$. We define the curve in the $(a,b)$ plane in the following way:
$$ C_{i}:\;\op{Disc}(F_{a,b,i}(x))=0. $$
Let $i=3$. The genus of the curve $C_{3}$ is equal to 3.
As a consequence of Faltings theorem we get that there are only finitely many required pairs $(a,b)$. Due to the identity $p_{3}(-x-3)=p_{3}(x)$ we can consider the points on $C_{3}$ with $a>0$ only. Using {\sc Magma} procedure {\tt PointSearch} we found that there are only six pairs of the required shape in the range $\op{max}\{H(a),H(b)\}\leq 10^5$, where $H(r)$ is the height of the rational number $r$. There are in total 16 rational points on $C_{3}$ in this range. More precisely, we have $(a,b)\in\cal{A}_{3}$, where
$$
\cal{A}_{3}=\left\{(1,1), (1,-3), (3,-1), (3,7), \left(\frac{1}{3},\frac{1}{3}\right), \left(\frac{1}{3},-\frac{7}{3}\right)\right\}.
$$
One can also observe that the pairs of points $(3,-1), (1/3,1/3)$ and $(3,7), (1/3, -7/3)$ lead to the same quadratic equations. We thus are left with the pairs $(1, 1), (1, -3), (3, -1)$ and $(3, 7)$.

If $(a,b)=(1,1)$, then $p_{3}(x)+p_{3}(x+1)=2(x+1)(x+3)(x+2)^2$. The solutions of the quadratic equation $v^2=2(x+2)^2-2$ are given by
$$
x=\frac{1}{2}((3+2\sqrt{2})^{n}+(3-2\sqrt{2})^n)-2,\quad v=\sqrt{2}((3+2\sqrt{2})^{n}-(3-2\sqrt{2})^n),
$$
and the corresponding value of $z$ is then $z=(x+2)v$.

Using exactly the same methods we cover the cases $(a,b)=(3,-1), (3,7)$. In the former case we deal with the equation $v^2=2(41x^2+30x+1)$, with non-trivial solution $(x,v)=(1,12)$ (the corresponding value of $z$ is $z=xv$). In the latter case we deal with the equation $2(41x^2+216x+280)=z^2$, with non-trivial solution $(x,v)=(4,60)$ (the corresponding value of $z=(x+3)v$. In both cases we get infinitely many positive solutions. We omit the standard details.

If $(a,b)=(1,-3)$ then $p_{3}(x)+p_{3}(x-3)=2(x^2+11)x^2$. However, 2 is a quadratic non-residue of 11 and we get no solutions.

If $i=4$ then
$$
\op{Disc}(F_{a,b,4}(x))=G_{1}(a,b)G_{2}(a,b),
$$
where
$$
G_{1}(a,b)=24u^4 - 100u^3v +105u^2v^2-40uv^3+5v^4, \mbox{where we put}\quad u=a+1, \quad v=b+4.
$$
The polynomial $G_{1}$ is irreducible and the unique solution of the equation $G_{1}(a,b)=0$ is given by $(a,b)=(-1,-4)$. Then $p_{4}(x)+p_{4}(-x-4)=0$ and we get infinitely many integer solutions but with $z=0$.

The second factor is a huge polynomial of degree 12 (with respect to each variable) which defines the curve, say $C_{4}$,
in the $(a,b)$ plane in the following way:
$$
C_{4}:\;G_{2}(a,b)=0.
$$
The genus of the curve $C_{4}$ is 3 (relatively small compared to the degree of the defining polynomial) and thus the set of rational points is finite. We used procedure {\tt PointSearch} one more time and find that the curve $C_{4}$ contains relatively many rational points. Indeed, in the range $\op{max}\{H(a),H(b)\}<10^5$ we found 16 rational points. We have $(a,b)\in\cal{A}_{4}$, where
\begin{eqnarray*}
&&\cal{A}_{4}=\{(-1,-8), (-1,-6), (-1,2), (-1,0), \left(\frac{1}{4},-\frac{15}{4}\right), \left(\frac{1}{4},-\frac{13}{4}\right),  \left(\frac{1}{4},\frac{1}{4}\right), \\
&&\hskip 1cm \left(\frac{1}{4},\frac{3}{4}\right), \left(\frac{2}{3},-\frac{5}{3}\right), \left(\frac{2}{3},\frac{1}{3}\right),  \left(\frac{3}{2},\frac{5}{2}\right), (4,-3), (4,-1), (4,13)\}.
\end{eqnarray*}
If $(a,b)\in\cal{A}_{4}\setminus\{(-1,0), (-1,-2)\}$, then the equation $z^2=p_{4}(x)+p_{4}(ax+b)$ defines a genus 1 curve (and thus there are only finitely many integer solutions in this case) or defines a genus 0 curve with only finitely many integral solutions.

In the first case we have $p_{4}(x)+p_{4}(-x)=20x^2(x^2+5)$, i.e., $5(x^2+5)$ needs to be square. Hence 5 divides $x$. Write $x=5t$. That is we obtain an equation of the form
$$
v^2-5t^2=4.
$$
The solutions of the above equation are well-known: $(v,t)=(L_{2m},F_{2m})$ for some $m\in\N$, where, as usual, $F_{n}$ and $L_{n}$ denotes the $n$-th Fibonacci and Lucas number, respectively, defined recursively $F_{0}=0, F_{1}=1, F_{n}=F_{n-1}+F_{n-2}$ and $L_{0}=2, L_{1}=1, L_{n}=L_{n-1}+L_{n-2}$.
 To obtain integral solution the number $F_{2m}$ has to be even, therefore $m$ is divisible by 3. It follows that if $t=5F_{6n},n\in\N$, then the pair
$$
\left(\frac{-5F_{6n}}{2},\frac{25F_{12n}}{2}\right)
$$
is a solution of the equation $z^2=p_{4}(x)+p_{4}(-x)$.

Using similar approach one can easily check that if $(a,b)=(-1,-2)$, then we get quadratic equation with infinitely many solutions. We omit the details.
\end{proof}	

\begin{rem}
{\rm Let us note the identities
$$
p_{4}(x)+p_{4}(-x-6)=-10(x+2)(x+4)(x+3)^2,\; p_{4}(x)+p_{4}(-x-8)=-20(x^2+8x+21)(x+4)^2,
$$
which can be used to prove that the equation $-z^2=p_{4}(x)+p_{4}(y)$ has infinitely many solutions in integers.
}
\end{rem}

\begin{rem}
	{\rm
		It seems that the question concerning the existence of {\it positive} integer solutions of the equation $z^2=p_{4}(x)+p_{4}(y)$ is more difficult. We performed numerical search for solutions in the range $0<x\leq y\leq 10^5$ and did not find any. This motivated us to formulate the following
		\begin{ques}
			Is the set of solutions in positive integers $x, y, z$ of the equation $z^2=p_{4}(x)+p_{4}(y)$ finite?
		\end{ques}
	}
\end{rem}

\noindent {\bf Acknowledgments}
The authors express their gratitude to the referee for careful reading of the manuscript
and valuable suggestions, which improved the quality of the paper.

This work was partially supported by the European Union and the European Social Fund through project EFOP-3.6.1-16-2016-00022 (Sz.T.). The research was supported in part by grants K115479 and K128088 (Sz.T.) of the Hungarian National Foundation for Scientiﬁc Research. The research is partially supported by the grant of the Polish National Science Centre no. DEC-2017/01/X/ST1/00407 (M.U.).

\bigskip \noindent Szabolcs Tengely, Institute of Mathematics, University of Debrecen, P.O.Box 12, 4010 Debrecen, Hungary. email:\;{\tt tengely@science.unideb.hu}

\bigskip \noindent Maciej Ulas, Jagiellonian University, Faculty of Mathematics and Computer Science, Institute of Mathematics, {\L}ojasiewicza 6, 30 - 348 Krak\'{o}w, Poland. e-mail:\;{\tt maciej.ulas@uj.edu.pl}

\end{document}